\documentclass[12pt]{amsart}
\usepackage{amsmath,amssymb,amsbsy,amsfonts,amsthm,latexsym,
                        amsopn,amstext,amsxtra,euscript,amscd,mathrsfs,color,bm}
                        
\usepackage{float} 
\usepackage[english]{babel}
\usepackage{url}
\usepackage[colorlinks,linkcolor=blue,anchorcolor=blue,citecolor=blue,backref=page]{hyperref}
\usepackage[norefs,nocites]{refcheck}

 \usepackage[np]{numprint}

\npdecimalsign{\ensuremath{.}}

\usepackage{bibentry}

\usepackage[english]{babel}
\usepackage{mathtools}
\usepackage{todonotes}
\usepackage{url}
\usepackage[colorlinks,linkcolor=blue,anchorcolor=blue,citecolor=blue,backref=page]{hyperref}

\begin{document}

\newcommand{\commA}[2][]{\todo[#1,color=yellow]{A: #2}}
\newcommand{\commI}[2][]{\todo[#1,color=green!60]{I: #2}}
    
\newtheorem{theorem}{Theorem}
\newtheorem{lemma}[theorem]{Lemma}
\newtheorem{example}[theorem]{Example}
\newtheorem{algol}{Algorithm}
\newtheorem{corollary}[theorem]{Corollary}
\newtheorem{prop}[theorem]{Proposition}
\newtheorem{definition}[theorem]{Definition}
\newtheorem{question}[theorem]{Question}
\newtheorem{problem}[theorem]{Problem}
\newtheorem{remark}[theorem]{Remark}
\newtheorem{conjecture}[theorem]{Conjecture}

\def\xxx{\vskip5pt\hrule\vskip5pt}

\def\Cmt#1{\underline{{\sl Comments:}} {\it{#1}}}

\newcommand{\Modp}[1]{
\begin{color}{blue}
 #1\end{color}}
 
 \def\bl#1{\begin{color}{blue}#1\end{color}} 
 \def\red#1{\begin{color}{red}#1\end{color}} 


\def\cA{{\mathcal A}}
\def\cB{{\mathcal B}}
\def\cC{{\mathcal C}}
\def\cD{{\mathcal D}}
\def\cE{{\mathcal E}}
\def\cF{{\mathcal F}}
\def\cG{{\mathcal G}}
\def\cH{{\mathcal H}}
\def\cI{{\mathcal I}}
\def\cJ{{\mathcal J}}
\def\cK{{\mathcal K}}
\def\cL{{\mathcal L}}
\def\cM{{\mathcal M}}
\def\cN{{\mathcal N}}
\def\cO{{\mathcal O}}
\def\cP{{\mathcal P}}
\def\cQ{{\mathcal Q}}
\def\cR{{\mathcal R}}
\def\cS{{\mathcal S}}
\def\cT{{\mathcal T}}
\def\cU{{\mathcal U}}
\def\cV{{\mathcal V}}
\def\cW{{\mathcal W}}
\def\cX{{\mathcal X}}
\def\cY{{\mathcal Y}}
\def\cZ{{\mathcal Z}}

\def\C{\mathbb{C}}
\def\F{\mathbb{F}}
\def\K{\mathbb{K}}
\def\L{\mathbb{L}}
\def\G{\mathbb{G}}
\def\Z{\mathbb{Z}}
\def\R{\mathbb{R}}
\def\Q{\mathbb{Q}}
\def\N{\mathbb{N}}
\def\M{\textsf{M}}
\def\U{\mathbb{U}}
\def\P{\mathbb{P}}
\def\A{\mathbb{A}}
\def\fp{\mathfrak{p}}
\def\n{\mathfrak{n}}
\def\X{\mathcal{X}}
\def\x{\textrm{\bf x}}
\def\w{\textrm{\bf w}}
\def\a{\textrm{\bf a}}
\def\k{\textrm{\bf k}}
\def\ee{\textrm{\bf e}}
\def\ovQ{\overline{\Q}}
\def \Kab{\K^{\mathrm{ab}}}
\def \Qab{\Q^{\mathrm{ab}}}
\def \Qtr{\Q^{\mathrm{tr}}}
\def \Kc{\K^{\mathrm{c}}}
\def \Qc{\Q^{\mathrm{c}}}
\newcommand \rank{\operatorname{rk}}
\def\ZK{\Z_\K}
\def\ZKS{\Z_{\K,\cS}}
\def\ZKSf{\Z_{\K,\cS_f}}
\def\ZKSfG{\Z_{\K,\cS_{f,\Gamma}}}

\def\bF{\mathbf {F}}

\def\({\left(}
\def\){\right)}
\def\[{\left[}
\def\]{\right]}
\def\<{\langle}
\def\>{\rangle}

\def\gen#1{{\left\langle#1\right\rangle}}
\def\genp#1{{\left\langle#1\right\rangle}_p}
\def\genPs{{\left\langle P_1, \ldots, P_s\right\rangle}}
\def\genPsp{{\left\langle P_1, \ldots, P_s\right\rangle}_p}

\def\e{e}

\def\eq{\e_q}
\def\fh{{\mathfrak h}}

\def\lcm{{\mathrm{lcm}}\,}

\def\({\left(}
\def\){\right)}
\def\fl#1{\left\lfloor#1\right\rfloor}
\def\rf#1{\left\lceil#1\right\rceil}
\def\mand{\qquad\mbox{and}\qquad}

\def\jt{\tilde\jmath}
\def\ellmax{\ell_{\rm max}}
\def\llog{\log\log}

\def\m{{\rm m}}
\def\ch{\hat{h}}
\def\GL{{\rm GL}}
\def\Orb{\mathrm{Orb}}
\def\Per{\mathrm{Per}}
\def\Preper{\mathrm{Preper}}
\def \S{\mathcal{S}}
\def\vec#1{\mathbf{#1}}
\def\ov#1{{\overline{#1}}}
\def\Gal{{\mathrm Gal}}
\def\Sp{{\mathrm S}}
\def\tors{\mathrm{tors}}
\def\PGL{\mathrm{PGL}}
\def\wH{{\rm H}}
\def\Gm{\G_{\rm m}}

\def\house#1{{%
    \setbox0=\hbox{$#1$}
    \vrule height \dimexpr\ht0+1.4pt width .5pt depth \dp0\relax
    \vrule height \dimexpr\ht0+1.4pt width \dimexpr\wd0+2pt depth \dimexpr-\ht0-1pt\relax
    \llap{$#1$\kern1pt}
    \vrule height \dimexpr\ht0+1.4pt width .5pt depth \dp0\relax}}

\newcommand{\bfalpha}{{\boldsymbol{\alpha}}}
\newcommand{\bfomega}{{\boldsymbol{\omega}}}

\newcommand{\Ch}{{\operatorname{Ch}}}
\newcommand{\Elim}{{\operatorname{Elim}}}
\newcommand{\proj}{{\operatorname{proj}}}
\newcommand{\h}{{\operatorname{\mathrm{h}}}}
\newcommand{\ord}{\operatorname{ord}}

\newcommand{\hh}{\mathrm{h}}
\newcommand{\aff}{\mathrm{aff}}
\newcommand{\Spec}{{\operatorname{Spec}}}
\newcommand{\Res}{{\operatorname{Res}}}

\def\fA{{\mathfrak A}}
\def\fB{{\mathfrak B}}

\numberwithin{equation}{section}
\numberwithin{theorem}{section}

\title[On the modulus of integers in Kummer extensions]
{On the maximum modulus of integers in Kummer extensions}

\author{Jorge Mello}
\address{School of Mathematics and Statistics, University of New South Wales, Sydney, NSW 2052, Australia}
\email{j.mello@unsw.edu.au}


\begin{abstract}  We study the extension of a result of Loxton (1972) on representation of algebraic integers as sums of roots of unity to Kummer extensions.
\end{abstract}

\maketitle

\section{Introduction}
For any algebraic number $\beta$, let $\house{\beta}$ be the maximum of the absolute values of the conjugates of $\beta$ over $\mathbb{Q}$. Suppose that $\beta$ is an algebraic integer contained in some cyclotomic field. Then, certainly, $\beta$ is a sum of (not necessarily distinct) roots of unity $\beta=\sum_{i=1}^b \xi_i$. A theorem of Loxton [5, Th. 1] shows that
one can choose the roots of unity, so that $b \leq L(\house{\beta})$ for a suitable function $L: \mathbb{R}_+ \rightarrow \mathbb{R}_+$. Loxton proves that there exists a positive constant $c$ such that
\begin{center}
$\house{\beta}^2 \geq cb \exp (- \log b/ \log \log b)$
\end{center} for every cyclotomic integer $\beta$ with $b \neq 0,1$. 

This result has found many applications in the literature,
as to obtain a proof for a cyclotomic version of Hilbert's irreducibility Theorem [3], finiteness of multiplicative dependent values of rational functions and iterated values of rational functions [9,10], and finiteness of preperiodic points of rational functions falling on the cyclotomic closure of a number field [5,8].

Here we seek to extend the ideas and results of Loxton to algebraic integers belonging to fields of decomposition of polynomials $X^N-a$.
Namely, let $a$ be a positive integer that is equal to $1$, or that is not a perfect power of any rational number, and let $\beta$ be an algebraic integer in some field of the form $\mathbb{Q}(\zeta_N, \sqrt[N]{a})$ where $\zeta_N$ is a $N$-th root of unity. We denote this field by $\mathbb{Q}_a(N)$.
Then $\beta$ can be represented as a sum of algebraic numbers $r_{ij}\zeta^i_N\sqrt[n]{a^j}$ with $r_{ij} \in \mathbb{Q} \cap [0,1]$.
Letting $N=p_1^{e_1}...p_s^{e_s}$ be the decomposition in primes of $N$ with $p_1 <p_2< ...< p_s$, and $\zeta_m$ denotes a $m$-th primitive root of unity for each $ m  \mid N$, so that $\zeta_m=\zeta_N^{N/m}$, 
it is a fact that each $r_{ij}$ can be chosen to be equal to the inverse of the rational number
\begin{align*}
\Delta_a(N):=\left|\text{Nm}_{\mathbb{Q}_a(N)|\mathbb{Q}}\left(\displaystyle\prod_{1 \leq i \leq s} \displaystyle\prod_{1 \leq t \leq e_i}\text{disc}_{\mathbb{Q}_a(p_1^{e_1}...p_i^t)/\mathbb{Q}_a(p_1^{e_1}...p_i^{t-1})}(\{\zeta_{p_i^t}^l \sqrt[p_i^t]{a^k} \}_{l,k})\right)\right|,
\end{align*} so that $\Delta_a(N) \beta$ can be represented as a sum of algebraic integers  of the form $\xi_i\alpha_j$, where $\xi_i$ is a root of unity and $\alpha_j$ is a positive real $n$-root of $a$. 
We denote by $M_{a,N}(\beta)$ the least number  of algebraic integers $\xi_i\alpha_j$ in this way occurring in any sum of this kind representing $\Delta_a(N)\beta$, which is the least number of complex roots of $a$ occuring in such sum representing $\Delta_a(N)\beta$. For such $N$, $M_{a,N}(\beta)$ is thus the smallest number of summands in a representation for $\beta$ of the form $\sum_{i,j}\xi_i\alpha_j/\Delta_a(N)$, counting repetition.

The object of the paper is the following 
\begin{theorem} Suppose that $k > \log 2$. Then there exists a positive number $c$ depending only on $k$ such that 
\begin{center}
$\Delta_a(N)^2\house{\beta}^2 \geq c M_{a,N}(\beta) \exp (-k \log M_{a,N}(\beta)/ \log \log M_{a,N}(\beta))$
\end{center} for all algebraic integers $\beta$ in some Kummer extension given by adjoining to the rationals the roots of $X^N-a$ for some $N>0$, with $M_{a,N}(\beta) \notin \{ 0,1 \}$, and  $\house{\beta}$ denoting the house of $\beta$.
\end{theorem} When $a=1$, the proofs and statement work with $1$ in place of $\Delta_a(N)$, recovering Loxton's result in this way.

In Section 2 we deal with an average for the squares of the conjugates of an algebraic integer in the studied extensions. In 3, we recall estimates and inequalities for the concave function $f(n)=n\exp(\frac{-\log n}{ \log \log n})$.
Section 4 contains intermediate results towards the proof of Theorem 1.1, which follows in Section 6.

\section{The function $\mathcal{M}$}
For any algebraic number $x$, we denote by $\mathcal{M}(x)$ the mean of $|x^{\prime}|^2$ taken over all the conjugates $|x^{\prime}|$ of $x$.  
As usual, $\house{x}$ denotes the maximun of the absolute values of $|x^{\prime}|$ of the conjugates $x^{\prime}$ of $x$, called the \textit{house} of $x$. Trivially
\begin{equation}
\house{x}^2 \geq \mathcal{M}(x).
\end{equation}
Also, if $x$ is a non-zero integer, its norm is at least $1$ in absolute value, and so
\begin{equation}
 \mathcal{M}(x) \geq 1,
\end{equation} by the inequality of the arithmetic and geometric means applied to the $|x^{\prime}|^2$.

For any integer $N \geq 1$ and rational number $a$, we denote by $\mathbb{Q}_a(N)$ the field obtained by adjoining all the roots of $X^N-a$ to the field of rationals $\mathbb{Q}$.

\textit{First case.} Suppose that $N=pN_1$, where $p$ is a prime and $p \nmid N_1$. 
For $\zeta_N$ a primitive $N$th root of unity, that can be chosen as the product $\zeta_P \zeta_{N_1}$ between a primitive $p$th root of unity and a $N_1$th root of unity, we have that $\mathbb{Q}_a(N)=\mathbb{Q}(\zeta_N, \sqrt[N]{a})$, where $\sqrt[N]{a}$ is a real $n$-root of $a$.
Then $[\mathbb{Q}_a(N): \mathbb{Q}]=[\mathbb{Q}(\zeta_N, \sqrt[N]{a}):\mathbb{Q}(\zeta_N)][\mathbb{Q}(\zeta_N): \mathbb{Q}]=N \phi(N)$. We note that $\{\sqrt[N_1]{a},\sqrt[p]{a}\} \subset \mathbb{Q}(\zeta_N, \sqrt[N]{a})$, with $[\mathbb{Q}(\zeta_N, \sqrt[N_1]{a}):\mathbb{Q}(\zeta_N)]=N_1$ and $[\mathbb{Q}(\zeta_N, \sqrt[p]{a}):\mathbb{Q}(\zeta_N)]=p$.
By Lemma 3.4.2 and Proposition 3.5.5 of [11], $\mathbb{Q}(\zeta_N, \sqrt[N_1]{a})$ and $\mathbb{Q}(\zeta_N, \sqrt[p]{a})$ are disjoint extensions of $\mathbb{Q}(\zeta_N)$, whose compositum has degree $N$ over $\mathbb{Q}(\zeta_N)$ and is equal to 
\begin{center}$\mathbb{Q}(\zeta_N, \sqrt[N_1]{a}+ \sqrt[p]{a})=\mathbb{Q}(\zeta_{N_1}).\mathbb{Q}(\zeta_p, \sqrt[N_1]{a}+ \sqrt[p]{a})$.\end{center} This compositum extension, by divisibility of degrees, must equal to $\mathbb{Q}(\zeta_N, \sqrt[N]{a})=\mathbb{Q}_a(N)$. 

Then any $\beta \in \mathbb{Q}_a(N)$ may be written in the form
\begin{equation}
 \beta = \displaystyle\sum_{\substack{0 \leq j \leq p-1 \\0 \leq i \leq p-1}} \alpha_{ij}\zeta^i \sqrt[p]{a^j}
\end{equation}
where $\zeta$ is a primitive $p$th root of unity, $\sqrt[p]{a}$ is some $p$th root of $a$ in $\overline{\mathbb{Q}}$, and $\alpha_{ij} \in \mathbb{Q}_a(N_1)$ $(0 \leq j \leq p-1, 0 \leq i \leq p-1)$.
Therefore, $\{ \zeta^i \sqrt[p]{a^j} \}_{0 \leq i \leq p-2, 0 \leq j \leq p-1}$ is a basis for the extension $\mathbb{Q}_a(N)/ \mathbb{Q}_a(N_1)$ formed by algebraic integers.
For $\beta \in \mathbb{Q}_a(N)$ an algebraic integer, Lemma 2.9 of [7] shows that disc$_{\mathbb{Q}_a(N)/ \mathbb{Q}_a(N_1)}.\beta$ can be represented by a linear combination of the basis $\{ \zeta^i \sqrt[p]{a^j} \}_{0 \leq i \leq p-2, 0 \leq j \leq p-1}$ with scalars that are algebraic integers in $\mathbb{Q}_a(N_1)$,

Thus, when $\beta$ is an algebraic integer, we can write
\begin{equation}
 \beta = \displaystyle\sum_{0 \leq i,j\leq p-1} \alpha_{ij} r_{ij} \zeta^i \sqrt[p]{a^j}
\end{equation} with $\alpha_{ij}$ algebraic integers in $\mathbb{Q}_a(N_1)$, and $r_{ij} \in \mathbb{Q}$. Moreover, $r_{ij}$ can be chosen all equal to $\dfrac{1}{|\text{Nm}_{\mathbb{Q}_a(N)|\mathbb{Q}}(\text{disc}_{\mathbb{Q}_a(N)/ \mathbb{Q}_a(N_1)}\{ \zeta^i \sqrt[p]{a^j} \}_{i,j})|}$ by the discussion above.

 We calculate $\mathcal{M}(\beta)$ over the conjugates over $\mathbb{Q}_a(N_1)$ by letting $\zeta$ in (2.4) run through all the primitive $p$-th roots of 1, and $\sqrt[p]{a}$ through all roots of $X^N -a$. We thus state and prove the following
 \begin{lemma}Let $N,N_1\geq 1$ be integers, and let $p$ be a prime such that $N=pN_1$ and $p \nmid N_1$. Let $\beta=\displaystyle\sum_{0 \leq i,j\leq p-1} \alpha_{ij} r_{ij} \zeta^i \sqrt[p]{a^j}$  an algebraic integer in $\mathbb{Q}_a(N)$, so that the $\alpha_{ij}$ can be chosen algebraic integers in $\mathbb{Q}_a(N_1)$ as above. Then \begin{center}
 $ \mathcal{M}(\beta)=\displaystyle\sum_{j \neq 0, i \geq 0 }\mathcal{M}(\alpha_{ij}r_{ij})\sqrt[p]{a^{2j}} +\dfrac{1}{2(p-1)}\displaystyle\sum_{i,j}\mathcal{M}(\alpha_{i0}r_{i0}- \alpha_{j_0}r_{j0}).$
 \end{center}
 \end{lemma} \begin{proof} We compute
\begin{align*}
 &p(p-1)\mathcal{M}(\beta)\\&=\displaystyle\sum_{\substack{1 \leq l \leq p-1 \\0 \leq k \leq p-1}} \left|\displaystyle\sum_{\substack{0 \leq j \leq p-1 \\0 \leq i \leq p-1}} \alpha_{ij}r_{ij}(\zeta^i)^l(\zeta^j)^k\sqrt[p]{a^j} \right|^2\\
 &=\displaystyle\sum_{\substack{1 \leq l \leq p-1 \\0 \leq k \leq p-1}} \displaystyle\sum_{i,j}|\alpha_{ij}|^2r^2_{ij}\sqrt[p]{a^{2j}}
 \\ &   +  \displaystyle\sum_{\substack{1 \leq l \leq p-1 \\0 \leq k \leq p-1}}\displaystyle\sum_{(i,j) \neq (n,m)} \alpha_{ij}\bar{\alpha}_{nm}r_{ij}r_{nm}(\zeta^i \bar{\zeta}^n)^l(\zeta^j \bar{\zeta}^m)^k \sqrt[p]{a^{j+m}}\\
 &= p(p-1)\displaystyle\sum_{i,j}|\alpha_{ij}|^2r^2_{ij}\sqrt[p]{a^{2j}} - p\displaystyle\sum_{i \neq n} \alpha_{i0}\bar{\alpha}_{n0}r_{i0}r_{n0}\\
 &=p(p-1)\displaystyle\sum_{j \neq 0,i}|\alpha_{ij}|^2r^2_{ij}\sqrt[p]{a^{2j}}\\
 &  + p(p-1)\displaystyle\sum_{i}|\alpha_{i0}|^2r^2_{i0} - p\displaystyle\sum_{i \neq n} \alpha_{i0}\bar{\alpha}_{n0}r_{i0}r_{n0}     \\
&=p\left( (p-1)\displaystyle\sum_{j \neq 0,i}|\alpha_{ij}|^2r^2_{ij}\sqrt[p]{a^{2j}} + \dfrac{1}{2}\displaystyle\sum_{i,j}|\alpha_{i0}r_{i0}- \alpha_{j_0}r_{j0}|^2   \right)\\
&=p\left( (p-1)\displaystyle\sum_{j \neq 0,i}\mathcal{M}(\alpha_{ij}r_{ij})\sqrt[p]{a^{2j}} + \dfrac{1}{2}\displaystyle\sum_{i,j}\mathcal{M}(\alpha_{i0}r_{i0}- \alpha_{j_0}r_{j0})   \right),
 \end{align*} and therefore
 \begin{equation}
  \mathcal{M}(\beta)=\displaystyle\sum_{j \neq 0,i}\mathcal{M}(\alpha_{ij}r_{ij})\sqrt[p]{a^{2j}} +\dfrac{1}{2(p-1)}\displaystyle\sum_{i,j}\mathcal{M}(\alpha_{i0}r_{i0}- \alpha_{j_0}r_{j0}).
 \end{equation} \end{proof}
 If precisely $I$ is the set of non-zero $\alpha_{ij}r_{ij}$, we may write (2.4) in the form
 \begin{equation}
  \beta= \displaystyle\sum_{(i,j) \in I} \gamma_{ij}t_{ij}\zeta^{i}\sqrt[p]{a^{j}},
 \end{equation} where  $\gamma_{ij} \in \mathbb{Q}_a(N_1), t_{ij}=\dfrac{1}{|\text{Nm}_{\mathbb{Q}_a(N)|\mathbb{Q}}(\text{disc}_{\mathbb{Q}_a(N)/ \mathbb{Q}_a(N_1)}\{ \zeta^i \sqrt[p]{a^j} \}_{i,j})|} \in \mathbb{Q}$. 
 
 Making $I_0:= \{ (i,j) \in I | s_j=0 \}$, we have, from (2.5),
 \begin{align*}
 &p(p-1)\mathcal{M}(\beta)\\&=\displaystyle\sum_{\substack{1 \leq l \leq p-1 \\0 \leq k \leq p-1}} \left|\displaystyle\sum_{(i,j) \in I} \gamma_{ij}t_{ij}(\zeta^{i})^l(\zeta^{j})^k\sqrt[p]{a^{j}} \right|^2\\
 &=\displaystyle\sum_{\substack{1 \leq l \leq p-1 \\0 \leq k \leq p-1}} \displaystyle\sum_{(i,j) \in I}|\gamma_{ij}|^2t^2_{ij}\sqrt[p]{a^{2j}}
 \\&   + \displaystyle\sum_{\substack{1 \leq l \leq p-1 \\0 \leq k \leq p-1}} \displaystyle\sum_{(i,j) \neq (n,m) \in  I_0} \gamma_{ij}\bar{\gamma}_{nm}t_{ij}t_{nm}(\zeta^{i} \bar{\zeta}^{n})^l(\zeta^{j} \bar{\zeta}^{m})^k \sqrt[p]{a^{j+m}}\\
 &= p(p-1)\displaystyle\sum_{(i,j) \in I}|\gamma_{ij}|^2t^2_{ij}\sqrt[p]{a^{2j}} - p\displaystyle\sum_{(i,j) \neq (n,m) \in I_0} \gamma_{ij}\bar{\gamma}_{nm}t_{ij}t_{nm}\\
 &=p(p-1)\displaystyle\sum_{(i,j) \notin I_0}|\gamma_{ij}|^2t^2_{ij}\sqrt[p]{a^{2j}}
 \\ &  + p(p-1)\displaystyle\sum_{(i,j) \in I_0}|\gamma_{ij}|^2t^2_{ij}\sqrt[p]{a^{2j}}- p\displaystyle\sum_{(i,j) \neq (n,m) \in I_0} \gamma_{ij}\bar{\gamma}_{nm}t_{ij}t_{nm} \\
&=p(p-1)\displaystyle\sum_{(i,j) \notin I_0}\mathcal{M}(\gamma_{ij}t_{ij})\sqrt[p]{a^{2j}}
\\ &  + p(p-|I_0|)\displaystyle\sum_{(i,j) \in I_0}\mathcal{M}(\gamma_{ij}t_{ij}) + \dfrac{p}{2}\displaystyle\sum_{(i,j) \in I_0}\mathcal{M}(\gamma_{ij}t_{ij}- \gamma_{nm}t_nm), \\
 \end{align*} which means in this case that 
 \begin{equation} \begin{array}{l}
   \mathcal{M}(\beta)=\displaystyle\sum_{(i,j) \in I-I_0}\mathcal{M}(\gamma_{ij}t_{ij})\sqrt[p]{a^{2j}}  + \dfrac{(p-|I_0|)}{(p-1)}\displaystyle\sum_{(i,j) \in I_0}\mathcal{M}(\gamma_{ij}t_{ij})
   \\ \qquad \quad  + \dfrac{1}{2(p-1)}\displaystyle\sum_{(i,j), (n,m) \in I_0}\mathcal{M}(\gamma_{ij}t_{ij}- \gamma_{nm}t_{nm}). \end{array}
 \end{equation}
 \textit{Second case.} Suppose that $N=p^LN_2$, where $p$ is a prime, $p \nmid N_2$ and $L \geq 2$. Put $N_1= p^{L-1}N_2$. 
 For $\zeta_N$ a primitive $N$th root of unity, it can be chosen for instance to be $\zeta_{p^L}\zeta_{N_2}$, where $\zeta_{p^L}$ and $\zeta_{N_2}$ are a $p^L$th primitive root of unity and a $N_2$th primitive root of unity respectively.
 Since $\zeta_{p^L}^p$ is a $p^{L-1}$th primitive root of unity, and $\sqrt[p^L]{a^p}=\sqrt[p^{L-1}]{a}$,
 every $\beta \in \mathbb{Q}_a(N)$ can be written in the form
 \begin{equation}
 \beta = \displaystyle\sum_{\substack{0 \leq j \leq p-1 \\0 \leq i \leq p-1}} \alpha_{ij}\zeta^i \sqrt[p]{a^j}
\end{equation} where $\zeta$ is a primitive $p^L$-th root of unity and $\alpha_{ij} \in \mathbb{Q}_a(N_1)$.
Again by Lemma 2.9 of [7], if $\beta$ is an algebraic integer, we can write
\begin{equation}
 \beta = \displaystyle\sum_{\substack{0 \leq j \leq p-1 \\0 \leq i \leq p-1}} \alpha_{ij} r_{ij} \zeta^i \sqrt[p]{a^j}
\end{equation} with $\alpha_{ij}$ algebraic integers in $\mathbb{Q}_a(N_1)$, and \begin{center}$r_{ij}=(\text{Nm}_{\mathbb{Q}_a(N)|\mathbb{Q}}(\text{disc}_{\mathbb{Q}_a(N)/ \mathbb{Q}_a(N_1)}\{ \zeta^i \sqrt[p]{a^j} \}_{i,j}))^{-1}$. \end{center}
Calculating $\mathcal{M}(\beta)$ over the conjugates over $\mathbb{Q}_a(N_1)$, we have this time the following easier formula
\begin{lemma}Let $N\geq 1$ be an integer, let $p$ be a prime, and let $L$ be the exponent of $p$ in $N$, with $L \geq 2$ and $N_1=N/p$ as above. Let $\beta=\displaystyle\sum_{0 \leq i,j\leq p-1} \alpha_{ij} r_{ij} \zeta^i \sqrt[p]{a^j}$  an algebraic integer in $\mathbb{Q}_a(N)$, so that the $\alpha_{ij}$ can be chosen algebraic integers in $\mathbb{Q}_a(N_1)$ as above. Then
\begin{center}
$\mathcal{M}(\beta)=\displaystyle\sum_{i,j}\mathcal{M}(\alpha_{ij}r_{ij})\sqrt[p^L]{a^{2j}}.$
\end{center}
\end{lemma}\begin{proof} We compute
\begin{align*}
 &p^2\mathcal{M}(\beta)\\&=\displaystyle\sum_{\substack{1 \leq l \leq p-1 \\0 \leq k \leq p-1}} \left|\displaystyle\sum_{\substack{0 \leq j \leq p-1 \\0 \leq i \leq p-1}} \alpha_{ij}r_{ij}(\zeta^i)^l(\zeta^j)^k\sqrt[p^L]{a^j} \right|^2\\
 &=\displaystyle\sum_{\substack{1 \leq l \leq p-1 \\0 \leq k \leq p-1}} \displaystyle\sum_{i,j}|\alpha_{ij}|^2r^2_{ij}\sqrt[p^L]{a^{2j}}\\
 &+ \displaystyle\sum_{\substack{1 \leq l \leq p-1 \\0 \leq k \leq p-1}}  \displaystyle\sum_{(i,j) \neq (n,m)} \alpha_{ij}\bar{\alpha}_{nm}r_{ij}r_{nm}(\zeta^i \bar{\zeta}^n)^l(\zeta^j \bar{\zeta}^m)^k \sqrt[p^L]{a^{j+m}}\\
&= p^2\displaystyle\sum_{i,j}\mathcal{M}(\alpha_{ij}r_{ij})\sqrt[p^L]{a^{2j}}.
 \end{align*} In this case
 \begin{equation}
  \mathcal{M}(\beta)=\displaystyle\sum_{i,j}\mathcal{M}(\alpha_{ij}r_{ij})\sqrt[p^L]{a^{2j}}.
 \end{equation} \end{proof}

\section{The functions $f$ and $g$}
Let $k >0$. Now we want to state some known facts about the function defined by
\begin{center}
$f(t)= f(t,k)=t\exp(-k\log t/\log \log t)$ for $t>0$ and $ t \neq 1$
\end{center} and 
\begin{center}
$f(0)=0, f(1)=1$.
\end{center}
We also consider the function
\begin{center}
$g(t)=t \exp(-k\log t^{\prime}/ \log \log t^{\prime})$
\end{center} where $t^{\prime}:= t + c_1$ and $c_1$ is a positive constant, possibly depending on $k$, which is to be chosen later. Now
\begin{center}
$g'(x)=\exp\left(-\dfrac{k\log t^{\prime}}{\log \log t'}\right) \left\{ 1 - \dfrac{kt}{t' \log \log t'} + \dfrac{kt}{t'(\log \log t')^2} \right\}$
\end{center} and
\begin{center}
$g''(t)=- \dfrac{k}{t' \log \log t'} \exp\left(-\dfrac{k\log t^{\prime}}{\log \log t'}\right) \left\{ 1 + O \left( \dfrac{1}{\log \log t'}  \right) \right\}$,
\end{center} with the constant implied by the $O$-notation depending only on $k$. So we can choose $k$ such that
\begin{equation}
g'(t) \geq 0 \textit{ and } g''(t) \leq 0 \textit{ for all } t \geq 0.
\end{equation}
and also 
\begin{equation} \log \log c_1 \geq 2.
\end{equation} This implies that $g$ is increasing and concave on $[0, \infty)$, and therefore by [4, Section 94], 
\begin{lemma}
If $a_1,..., a_\nu$ are non-negative real numbers, then
\begin{center}
$\dfrac{1}{\nu} \displaystyle\sum_{r=1}^\nu g(a_r)\leq g\left( \dfrac{1}{\nu} \displaystyle\sum_{r=1}^\nu a_r \right)$.
\end{center}
\end{lemma} The next lemma is also a consequence of concavity.
\begin{lemma}
Let $0 \leq \nu, \mu < \infty$ and $a >0$ be given. For any numbers $a_i,...,a_\nu$ satisfying
\begin{center}
$\lambda \leq a_r \leq \mu$ $(1 \leq r \leq \nu)$ and $\displaystyle\sum_{r=1}^\nu a_r \geq a$,
\end{center} we have
\begin{center}
$\displaystyle\sum_{r=1}^\nu g(a_r) \geq ug(\lambda) + (\nu - u -1)g(\mu) + g(\sigma)$
\end{center} where
\begin{center}
$u=\lfloor{(\mu \nu - a)/(\mu - \lambda)}\rfloor$ and $\sigma=a - u\lambda - (\nu - u -1)\mu.$
\end{center}
\end{lemma}
\begin{proof}
 See [5], Lemma 2.
\end{proof}

\begin{lemma}
If $t \geq c_1$, then
\begin{center}
$0 < \log f(t) - \log g(t) < \dfrac{c_1 k}{t \log \log t}$.
\end{center}
\end{lemma}
\begin{proof}
 See [5], Lemma 3
\end{proof}

\begin{lemma}
$g(s) + g(t) \geq g(s+t)$ if $s,t \geq 0$; further
\begin{center}
$g(s)+g(t) \geq g(s+t) + \dfrac{c_2g(t)}{\log \log t'}$ if $1 \leq t \leq s$,
\end{center} where
\begin{center}
$c_2 = \dfrac{k}{2(1 +c_1)}$.
\end{center}
\end{lemma}
\begin{proof}
 See [5], Lemma 4.
\end{proof}

\begin{lemma}
If $a_1,...,a_\nu$ are non-negative integers, then
\begin{center}
$\displaystyle\sum_{r=1}^\nu g(a_r)\geq g\left(  \displaystyle\sum_{r=1}^\nu a_r \right)$.
\end{center}
\end{lemma}
\begin{proof}
 See [5], Corollary of Lemma 4.
\end{proof}

\begin{lemma}
Let $k$ and $\delta$ be given positive numbers, with $k > \log 2$. Let $0 \leq t \leq s$ and put $u=s+t$. Then there is a positive number $c_3=c_3(k, \delta)$, depending only on $k$ and $\delta$, such that
\begin{center}
$g\left( \dfrac{t}{(\log t)^\delta} \right) \leq \dfrac{c_2 g(t)}{2 \log \log t'}$ whenever $t \geq c_3$,
\end{center}and 
\begin{center}
$g(s) + g(t) \geq g(u) + g \left(  \dfrac{u}{\log u} \right)$
\end{center} whenever $t \geq \max \{ c_3, u(\log u)^{\delta -1} \}$.

If, in addition, $\delta < 1 - k^{-1}\log 2$, then there is a positive number $c_4= c_4(k,\delta)$, depending only on $k$ and $\delta$, such that 
\begin{center}
$tg(\dfrac{s}{t})\geq 2g(s)$
\end{center} whenever
\begin{center}
$s \geq c_4$ and $\dfrac{1}{4}(\log s)^{1- \delta} \leq t \leq s^{1/2}$.
\end{center} \end{lemma}
\begin{proof}
 See [5], Lemma 5.
\end{proof}

\section{A basic inequality}
Throughout this section, $\beta$ denotes an algebraic integer in the fixed field $\mathbb{Q}_a(N)$.
We consider only the first case of section 2, namely, $N=pN_1$, where $p$ is a prime and $p \nmid N_1$.
As in section 2, $\zeta$ denotes a primitive $p$-th root of unity.

To shorten the notation, when dealing with the representations (2.3) and (2.6), we shall write
\begin{center}
 $a_{ij}:= \alpha_{ij}r_{ij}$ and $ b_{ij}:= \gamma_{ij}t_{ij}$.
\end{center} with $r_{ij}$ and $t_{ij}$ all equal  to $|\text{Nm}_{\mathbb{Q}_a(N)|\mathbb{Q}}(\text{disc}_{\mathbb{Q}_a(N)/ \mathbb{Q}_a(N_1)}\{ \zeta^i \sqrt[p]{a^j} \}_{i,j})|^{-1}$ as in Section 2.

Moreover,
\begin{center}
 $M_{a,N}(\beta)=n$, 
\end{center}
\begin{center}
 $M_{a,N_1}(\alpha_{ij})=n_{ij}$,  $M_{a,N_1}(\alpha_{ij}- \alpha_{kl})=n_{ijkl}$ $(0 \leq i,j,k,l \leq p-1)$,
\end{center}
\begin{center}
 $M_{a,N_1}(\gamma_{ij})=m_{ij}$,  $M_{a,N_1}(\gamma_{ij}- \gamma_{kl})= m_{ijkl}$ $((i,j),(k,l) \in I)$.
\end{center}
\begin{lemma}
 If $\beta=\displaystyle\sum_{(i,j) \in I}b_{ij}\zeta^{i}\sqrt[p]{a^{j}}$ and $|I| \leq \dfrac{1}{2}p(p-1)$, then \begin{center}$n= \displaystyle\sum_{(i,j) \in I} m_{ij}$.\end{center}
\end{lemma}
\begin{proof}
 Clearly, $n \leq \sum m_{ij}$. Suppose that $n <\sum m_{ij}$. Choose a representation $\beta = \displaystyle\sum_{i,j}a_{ij} \zeta^i \sqrt[p]{a^j}$ of the form (2.3) with $\sum_{i,j} n_{ij}=n$.
 
 Making $I_j= \{ i | 0 \leq i \leq p-1 , (i,j) \in I \}$, we have that
 \begin{center}
  $\beta= \displaystyle\sum_{l=0}^{p-1}\left(\displaystyle\sum_{i=0}^{p-1}a_{ij} \zeta^i \right) \sqrt[p]{a^l}=\displaystyle\sum_j \left( \displaystyle\sum_{i \in I_j} b_{ij} \zeta^{i}  \right)\sqrt[p]{a^{j}}$.  
 \end{center}
Since $\{ \sqrt[p]{a^j}| j=0,...,p-1 \}$ is a set of linear independent algebraic numbers, for each $j$ we have that
\begin{center}
 $\displaystyle\sum_{i=0}^{p-1}a_{ij} \zeta^i= \displaystyle\sum_{i \in I_j} b_{ij} \zeta^{i}$,
\end{center}
and then there is an $a_j$ such that
 \begin{center} $  
a_{ij} = 
     \begin{cases}
       \ b_{ij} + a_j
       \ &\quad\text{if} ~ i \in I_j, \\
       \
       \
       \ a_j
       \ &\quad\text{if} ~ i \notin I_j. \\

     \end{cases}
 $ \end{center}
 By the choice of $a_{ij}$, we have that $M_{a,N_1}(\displaystyle\sum_i\alpha_{ij}\zeta^i)= \displaystyle\sum_{i}n_{ij}$, and making $\alpha_j:=a_j/r_{ij}=a_j.|\text{Nm}_{\mathbb{Q}_a(N)|\mathbb{Q}}(\text{disc}_{\mathbb{Q}_a(N)/ \mathbb{Q}_a(N_1)}\{ \zeta^i \sqrt[p]{a^j} \}_{i,j})|$, we have
 \begin{align*}
 \displaystyle\sum_{i \in I_j}n_{ij} + |I_j|M_{a,N_1}(\alpha_j) &\geq  \sum_{i} n_{ij}=M_{a,N_1}(\displaystyle\sum_i\alpha_{ij}\zeta^i)\\ &\geq \displaystyle\sum_{i \in I_j}n_{ij} + (p-|I_j|)M_{a,N_1}(\alpha_j),
 \end{align*}
and therefore $|I_j| M_{a,N_1}(\alpha_j) > (p-|I_j|-1)M_{a,N_1}(\alpha_j)$, $|I_j| > \dfrac{p-1}{2}$, and hence $|I| > p(p-1)/2$, which is a contradiction with the hypothesis.
 \end{proof}
\begin{lemma}
 Let $k > \log2$ and $\beta=\displaystyle\sum_{(i,j) \in I} b_{ij}\zeta^i\sqrt[p]{a^j}$. If
 \begin{center}
  $|I| \leq \dfrac{1}{2}p(p-1) \min \{ 1, c_2/\log \log n' \}$,
 \end{center}then
\begin{equation}
 (p^2-|I|)\displaystyle\sum_{(i,j) \in I}g(m_{ij}) +\dfrac{1}{2} \displaystyle\sum_{(i,j),(k,l) \in I}g(m_{ijkl})\geq p(p-1)g(n).
\end{equation}

\end{lemma}
\begin{proof}
The statement holds easily for $|I|=1$. We suppose it holds for $J, |J| \geq 1$, and consider $I \supset J$ with $|I|=|J|+1$ satisfying the hypothesis of the lemma.
We consider
\begin{center}
 $\beta=\displaystyle\sum_{(i,j) \in I} b_{ij}\zeta^i\sqrt[p]{a^j}, \beta_1=\displaystyle\sum_{(i,j) \in J} b_{ij}\zeta^i\sqrt[p]{a^j}$.
\end{center}
Without loss of generality, we may suppose $I= \{(c,d)\} \cup J$ and $m_{cd} = \min \{m_{ij}| (i,j) \in I \}$. By Lemma 4.1, $M_{a,N}(\beta)=\displaystyle\sum_{(i,j) \in I} m_{ij}$ and
$M_{a,N}(\beta_1)= \displaystyle\sum_{(i,j) \in J} m_{ij}$. In particular $m_{cd} \leq m \leq n$. Also, $m_{ijcd}\geq m_{ij} - m_{cd}$, and so by Lemma 3.4, $g(m_{ijcd}) \geq  g(m_{ij}) - g(m_{cd})$.
Writing $T(I)$ for the left-hand side of (4.1), we have
\begin{align*}
 T(I)&=T(J) + (p^2-|J|-1)g(m_{cd}) +\displaystyle\sum_{(i,j) \in J}\{g(m_{ijcd})- g(m_{ij})  \}\\ 
 &\geq p(p-1)g(m) + (p^2-2|J| -1)g(m_{cd}), \text{by induction hypothesis }\\
 & \geq p(p-1)\left\{ g(n) + \dfrac{c_2g(m_{cd})}{\log \log m'_{cd}} \right\} -2|J| g(m_{cd}), \text{ by Lemma 3.4, }\\
 & \geq p(p-1)g(n), \text{ by hypothesis.}
 \end{align*} \end{proof}
 \begin{lemma}
  Let $\beta=\displaystyle\sum_{0\leq i,j\leq p-1} a_{ij}\zeta^i\sqrt[p]{a^j}$. Suppose that for each fixed $0\leq i,j\leq p-1$, at least $2g(n)/g(1)$ of the numbers $a_{ij}- a_{lk}$ $(0\leq l,k\leq p-1$ are non-zero. Then
  \begin{center}
   $\displaystyle\sum_{0\leq i,j,l,k\leq p-1}g(n_{ijlk}) \geq 2p(p-1)g(n)$.
  \end{center}
\end{lemma}
\begin{proof} Since $g$ is increasing, we have
\begin{center}
 $\displaystyle\sum_{0\leq i,j,l,k\leq p-1}g(n_{ijlk}) \geq \displaystyle\sum_{0\leq i,j \leq p-1} \dfrac{2g(n)}{g(1)}g(1)=2p^2g(n) \geq 2p(p-1)g(n)$.
\end{center} \end{proof} As a step towards our purposes, we have
\begin{theorem} Let $k > \log 2$. There is a positive number $c_5=c_5(k)$, depending only on $k$ with the following property. Suppose that as in (2.3), $\beta = \displaystyle\sum_{i,j}a_{ij} \zeta^i \sqrt[p]{a^j}$, where $a_{ij}= \alpha_{ij}r_{ij}$ with $\alpha_{ij}$ algebraic integers in $\mathbb{Q}_a(N_1)$, that $\log M_{a,N}(\beta) \leq p(p-1)$ and $p \geq c_5$. Then
 \begin{equation}
 \displaystyle\sum_{0\leq i,j,l,k \leq p-1}g[M_{a,N_1}(\alpha_{ij}-\alpha_{lk})] \geq 2p(p-1)g[M_{a,N}(\beta)].
\end{equation}
\end{theorem}
\begin{proof}
First, we may choose any representation of the form (2.3) for $\beta$, so we may suppose $\beta = \displaystyle\sum_{i,j}a_{ij} \zeta^i \sqrt[p]{a^j}$ and $\sum_{i,j} n_{ij}=n=M_{a,N}(\beta)$.
 
Next, for each fixed $j$,  permutations of the $ a_{ij}$ $ (i \in \{0,...,p-1  \})$ do not change $M_{a,N}(\beta)$. For $\sigma_j$ a permutation of $\{0,...,p-1\}$ and $\tau_j$ its inverse. Let $\beta^*:=\displaystyle\sum_{i,j}a_{\sigma_j(i)j} \zeta^i \sqrt[p]{a^j}$ and choose a representation $\beta^*=\displaystyle\sum_{i,j}a^*_{ij} \zeta^i \sqrt[p]{a^j}$, with
$\alpha_{ij}^*r_{ij}=a^*_{ij}$ and $\sum_{i,j}M_{a,N_1}(\alpha^*_{ij})= M_{a,N}(\beta^*)$. We see that there exist $a_j, j=0,...,p-1$, such that
\begin{center}
 $a^*_{ij}=a_{\sigma_j(i)j} + a_j, (0\leq j \leq p-1)$.
\end{center} Now
\begin{center}
 $\beta=\displaystyle\sum_{i,j}a_{ij} \zeta^i \sqrt[p]{a^j}=\displaystyle\sum_{i,j}(a_{ij}+ a_j) \zeta^i \sqrt[p]{a^j}=\displaystyle\sum_{i,j}a^*_{\tau_j(i)j} \zeta^i \sqrt[p]{a^j}$.
\end{center}So
\begin{center}
 $M_{a,N}(\beta) \leq \sum_{ij}M_{a,N_1}(\alpha^*_{ij})=M_{a,N}(\beta^*) \leq \sum_{ij}M_{a,N_1}(\alpha_{ij})=M_{a,N}(\beta)$,
\end{center} and hence $M_{a,N}(\beta)=M_{a,N}(\beta^*)$. So we may suppose 
\begin{equation}
 n_{0j} \geq n_{1j} \geq ...\geq n_{|I_j|-1,j} > n_{|I_j|j}=...= n_{p-1, j}=0.
\end{equation} We also choose $\delta=\delta(k)$, depending only on $k$, such that
\begin{equation}
 0< \delta < 1-k^{-1}\log 2.
\end{equation} Now we proceed by induction on $n$. If $n=0$, (4.2) is trivially true. So we make the following induction hypothesis: If
$\beta=\displaystyle\sum_{i,j}a^*_{ij} \zeta^i \sqrt[p]{a^j} \in \mathbb{Q}_a(N), a^*_{ij} \in \mathbb{Q}_a(N_1)$ in the statement setting, and $M_{a,N}(\beta^*)<n$, then
\begin{center}
 $\displaystyle\sum_{0\leq i,j,l,k \leq p-1}g[M_{a,N_1}(\alpha^*_{ij}-\alpha^*_{lk})] \geq 2p(p-1)g[M_{a,N}(\beta^*)]$.
\end{center}Now, to prove (4.2) for $n >0$, we distinguish three cases.

\textit{First case.} $\dfrac{4g(n)}{p(p-1)g(1)} \leq \min \left\{ 1, \dfrac{c_2}{\log \log n'}\right\}$. 

If any of the representations $\beta=\displaystyle\sum_{i,j}(a_{ij}- a_{lk}) \zeta^i \sqrt[p]{a^j}, (0 \leq l,k \leq p-1)$ has less than $\dfrac{1}{2}p(p-1)\min \left\{ 1, \dfrac{c_2}{\log \log n'}\right\}$ non-zero terms, then (4.2) follows from Lemma
4.2. Otherwise, all such representations have at least $2g(n)/g(1)$ non-zero terms, and (4.2) follows from Lemma 4.3. This proves the first case. From now on, we therefore suppose that
\begin{center}
 $\dfrac{4g(n)}{p(p-1)g(1)}> \min \left\{ 1, \dfrac{c_2}{\log \log n'}\right\}$.
\end{center} Consequently, there is a positive number $c_6=c_6(k)$, dependding only on $k$, such that
\begin{equation}
 p \leq  \min \{ n, n/\log n\} \text{ whenever } p\geq c_6.
\end{equation}

\textit{Second case} $n_{ij} \leq n(\log n)^{\delta -1} (0 \leq i , j \leq p-1)$. 

Set $t=\lfloor{\frac{1}{2}(\log n)^{1- \delta}}\rfloor$ and consider a fixed $(i,j), i,j \in \{ 0,...,p-1 \}$. Let $a_1,...,a_\nu$ be the non-zero numbers among the $n_{ijlk} (0 \leq l,k \leq p-1)$. Then
\begin{center}
 $\lambda:=1 \leq a_r \leq \mu:= nt^{-1}$ $(1 \leq r \leq \nu)$
\end{center} and 
\begin{center}
 $\displaystyle\sum_{i=1}^\nu a_r \geq n$.
\end{center} From this, $n \leq \sum a_r \leq \nu \max a_r \leq \nu nt^{-1}$, so
\begin{equation}
 \nu \geq t.
\end{equation} Now by (4.5), there is a number $c_7=c_7(k)\geq c_6$ such that
\begin{equation}
 p \leq nt^{-1} \text{ and } t \geq 2 \text{ whenever } p \geq c_7.
\end{equation} So, if $p \geq c_7$, 
\begin{center}
 $\left\lfloor{\dfrac{\mu \nu - n}{\mu - \lambda}}\right\rfloor=\left\lfloor{\nu - t + \dfrac{ \nu - t}{n/t -1}}\right\rfloor=\nu - t$
\end{center}
because by (4.6) and (4.7),
\begin{center}
 $0 \leq \dfrac{ \nu - t}{n/t -1} \leq \dfrac{ p - t}{n/t -1} \leq \dfrac{ n/t - t}{n/t -1} < 1$.
\end{center}Hence, by Lemma 3.2,
\begin{align*}
 \displaystyle\sum_{0 \leq l,k \leq p-1}g(n_{ijlk})&=\displaystyle\sum_{i=1}^\nu g(a_r)\\
 & \geq (\nu -t)g(1) + (t-1)g(nt^{-1}) + g(nt^{-1 - \nu +t})\\
 &\geq (t-1)g(nt^{-1}) + g(n^{t-1}) \text{ by Lemmas 3.4 and 3.5,}\\
 &=tg(nt^{-1})\\
 &\geq 2g(n), \text{ by Lemma 3.6},
\end{align*} providing $n \geq c_4$ and $\frac{1}{4}(\log n)^{1- \delta} \leq t \leq n^{1/2}$, which is true for  $p \geq c_7$. The inequalities above imply that 
\begin{center}
 $\displaystyle\sum_{0 \leq i,j,l,k \leq p-1}g(n_{ijlk}) \geq 2p(p-1)g(n)$ whenever $p \geq \max\{c_4. c_7 \}$. 
\end{center}
\textit{Third case.} $n_{0s}:= \max \{ n_{0j}| j=0,...,p-1 \} > n(\log n)^{\delta-1}$. 

Put 
\begin{center}
 $\beta_1=\displaystyle\sum_{0\leq i\leq p-1, j\neq s} a_{ij}\zeta^i\sqrt[p]{a^j} + \displaystyle\sum_{1\leq i \leq p-1}a_{is}\zeta^i\sqrt[p]{a^s}= \beta - a_{0s}\sqrt[p]{a^s}$.
  \end{center} We see that $M_{a,N}(\beta_1)= n - n_{0s}=m$.
  Thus the induction hypothesis applies to $\beta_1$ giving
  \begin{equation}
   \displaystyle\sum_{(i,j), (l.k) \neq (0,s)}g(n_{ijlk}) + 2 \displaystyle\sum_{(i,j) \neq (0,s)}g(n_{ij}) \geq 2p(p-1)g(m).
  \end{equation} By Lemma 3.4 again, we have $g(n_{0slk})\geq g(n_{0s}) - g(n_{lk})$. Using this and (4.8),
\begin{align*}
\displaystyle\sum_{0 \leq i,j,l,k \leq p-1}g(n_{ijlk}) &\geq 2p(p-1)g(m) + 2 \displaystyle\sum_{(l,k) \neq (0,s)} \{ g(n_{0slk})- g(n_{lk}) \}\\
&\geq 2p(p-1)\{ g(m) + g(n_{0s})\} -4 \displaystyle\sum_{(l,k) \neq (0,s)}g(n_{lk})\\
&\geq 2p(p-1) \left\{ g(m) + g(n_{0s}) -2g\left(\dfrac{m}{p(p-1)}\right) \right\},\\
&\text{ by Lemma 3.1. }
\end{align*}
\textit{First subcase.} $n_{0s} \leq \frac{1}{2}n$.

By hypothesis, $\log n \leq p(p-1)$. Also, $g$ is increasing, so the inequalities above imply that
\begin{align*}
\displaystyle\sum_{0 \leq i,j,l,k \leq p-1}g(n_{ijlk}) &\geq 2p(p-1)g(m) + 2 \displaystyle\sum_{(l,k) \neq (0,s)} \{ g(n_{0slk})- g(n_{lk}) \}\\
&\geq 2p(p-1) \left\{ g(m) + g(n_{0s}) -2g\left(\dfrac{n}{\log n}\right) \right\},\\
&\geq 2p(p-1)g(n) \text{ by Lemma 3.6, providing  } n_{0s} \geq c_3.
\end{align*}
But by (4.7), $n_{0s}> \dfrac{n}{(\log n)^{1- \delta}}> \dfrac{n}{\log n} \geq p$, whenever $p \geq c_7( \geq c_3)$

\textit{Second subcase. } $n_{0s} \geq \frac{1}{2}n$.

If  $|I|\leq  \frac{1}{2}p (p-1)\{ 1, c_2/\log \log n'  \} $, then (4.2) follows immediately from Lemma 4.2. So we can suppose that 
\begin{center}
$|I| > \frac{1}{2}p(p-1) \min \{ 1, c_2/\log \log n' \} \geq \frac{1}{2} \log n \min \{1, \log \log n' \}$.
\end{center} Now, $m \geq |I|$, so there is a number $c_8=c_8(k)$ such that
\begin{equation}
m \geq c_3(k) \text{ whenever } n \geq c_8.
\end{equation} Next, $p(p-1) \geq \log n \geq \log m$, so by the inequalities in the end of the first subcase and by Lemma 3.4,
\begin{align*}
\displaystyle\sum_{0 \leq i,j,l,k \leq p-1} g(n_{ijlk})&\geq 2p(p-1) \left\{ g(n) + \dfrac{c_2g(m)}{\log \log m'} -2g\left( \dfrac{m}{\log m} \right)\right\}\\
&\geq 2p(p-1)g(n) \text{ if } n \geq c_8, \text{ by (4.9) and Lemma 3.6.} 
\end{align*} Combining the three cases,
\begin{center}
$\displaystyle\sum_{0 \leq i,j,l,k \leq p-1} g(n_{ijlk})\geq 2p(p-1)g(n)$ whenever $p \geq c_5:= \max \{c_3, c_4, c_7, c_8 \}$.
\end{center} So the theorem follows by induction.
\end{proof}
\section{ Proof of Theorem 1.1}
We define the following
\begin{definition} Let $N=p_1^{e_1}...p_s^{e_s}$ be the decomposition in primes of $N$ with $p_1 <p_2< ...< p_s$ and $\zeta_m$ denotes a $m$-th primitive root of unity for each $ m  \mid N$, so that $\zeta_m=\zeta_N^{N/m}$.  
\begin{align*}
\Delta_a(N):=\left|\text{Nm}_{\mathbb{Q}_a(N)|\mathbb{Q}}\left(\displaystyle\prod_{1 \leq i \leq s} \displaystyle\prod_{1 \leq t \leq e_i}\text{disc}_{\mathbb{Q}_a(p_1^{e_1}...p_i^t)/\mathbb{Q}_a(p_1^{e_1}...p_i^{t-1})}(\{\zeta_{p_i^t}^l \sqrt[p_i^t]{a^k} \}_{l,k})\right)\right|,
\end{align*}
\end{definition}
Before stating and proving the main result of the paper, we state and prove a lemma that enables us to deal with the case when $N$ is a product of distinct small primes.
\begin{lemma}
 Denote the sequence of odd primes by $\{ p_r \}$ and put $p_0=1$. Suppose that $N=p_1p_2...p_\mu$ and let $\beta$ be an algebraic integer in $\mathbb{Q}_a(N)$. Then
 \begin{center}
  $\Delta_a(N)^2 \mathcal{M}(\beta)\geq 2^{-\mu}M_{a,N}(\beta)$.
 \end{center}

\end{lemma}
\begin{proof}
 The statement is true for $\nu=0$, because then $N=1$ and $\beta$ is an integer number, so $|\beta|=M_{a,1}(\beta)$ and $\mathcal{M}(\beta)=M_{a,1}(\beta)^2 \geq M_{a,1}(\beta)$. Suppose that the statement is true when $\nu=\mu -1 ( \mu \geq 1)$.
 Let $N=p_1p_2...p_\mu$ and $\beta \in \mathbb{Q}_a(N)$. Set $p=p_\mu$ and $\zeta$ be a primitive $p$th root of unity.
 Then we can write
 \begin{center}
  $\beta=\displaystyle\sum_{0 \leq i,j \leq p-1}a_{ij} \zeta^i \sqrt[p]{a^j}$.
 \end{center} where $a_{ij}=\alpha_{ij}r_{ij}$ and $\alpha_{ij} \in \mathbb{Q}_a(N/p)$ are algebraic integers as in (2.4) and $M_{a,N}(\beta)=\sum_{i,j}n_{ij}$.

 Now, by Lemma 2.1 and the induction hypothesis,
 \begin{align*}
  &\Delta_a(N)^2\mathcal{M}(\beta)\\&=\Delta_a(N)^2\left(\displaystyle\sum_{j \neq 0, i \geq 0 }\mathcal{M}(\alpha_{ij}r_{ij})\sqrt[p]{a^{2j}} +\dfrac{1}{2(p-1)}\displaystyle\sum_{i,j}\mathcal{M}(\alpha_{i0}r_{i0}- \alpha_{j_0}r_{j0})\right)\\
  &=\Delta_a(N_1)^2\left(\displaystyle\sum_{j \neq 0, i \geq 0 }\mathcal{M}(\alpha_{ij})\sqrt[p]{a^{2j}} +\dfrac{1}{2(p-1)}\displaystyle\sum_{i,j}\mathcal{M}(\alpha_{i0}- \alpha_{j_0})\right)\\
 &\geq 2^{-\mu+1} \left\{ \displaystyle\sum_{j \neq 0, i \geq 0 }M_{a,N_1}(\alpha_{ij}) +\dfrac{1}{2(p-1)}\displaystyle\sum_{i,j}M_{a,N_1}(\alpha_{i0}- \alpha_{j_0}) \right\}\\
 & \geq 2^{-\mu+1} \left\{ \displaystyle\sum_{j \neq 0, i \geq 0 }M_{a,N_1}(\alpha_{ij}) +\frac{1}{2(p-1)}\displaystyle\sum_{j}M_{a,N_1}\left(\sum_i\alpha_{i0}\zeta^i\right) \right\}\\
 & \geq 2^{-\mu+1} \left\{ \displaystyle\sum_{j \neq 0, i \geq 0 }M_{a,N_1}(\alpha_{ij}) +\frac{1}{2}\displaystyle M_{a,N_1}\left(\sum_i\alpha_{i0}\zeta^i\right) \right\}\\  
 & = 2^{-\mu} \displaystyle\sum_{i,j}M_{a,N_1}(\alpha_{ij})= 2^{-\mu}M_{a,N}(\beta).
 \end{align*}
 Hence, the statement is true by induction.
\end{proof}

\begin{proof} \textit{of Theorem 1.1}
  In order to prove the theorem for a given $k > \log 2$, it suffices to show that there is a positive number $c_9=c_9(k)$ such that, for all algebraic integers $\beta$ in Kummer extensions $\mathbb{Q}_a(N)$,
  \begin{equation}
   \Delta_a(N)^2\mathcal{M}(\beta) \geq c_9 g[M_{a,N}(\beta)].
  \end{equation} For suppose (5.1) holds. Let
  \begin{center}
$c:=c(k)= \min \{ c_9 \exp(-\frac{1}{2}k), f(1)^{-1},...,f(\lfloor{c_1}\rfloor)^{-1}  \}$,                                                  
 \end{center}  so that $c >0$. If $M_a(\beta) \geq c_1$, then by Lemma 3.3,
 \begin{align*}
  \Delta_a(N)^2\mathcal{M}(\beta) \geq c_9 g[M_{a,N}(\beta)]&\geq c_9 \exp(-\frac{1}{2}k)f[M_{a,N}(\beta)]\\&\geq c f[M_{a,N}(\beta)].
 \end{align*} If $0 \leq M_a(\beta) < c_1$, the same conclusion follows from (2.2) and the definition of $c$. So by (2.1), 
\begin{center}
 $\Delta_a(N)^2\house{\beta}^2 \geq \Delta_a(N)^2 \mathcal{M}(\beta)\geq \mathcal{M}(\beta) \geq c f[M_{a,N}(\beta)]$.
\end{center} It now remains to prove (5.1). To do this, we suppose that (5.1) is false for every $c_9>0$ and show that for suitable $c_9$ this leads to a contradiction. Choose $c_9$ initially with
\begin{equation}
 0< c_9 \leq 1.
\end{equation} Let $N$ be the smallest positive integer such that $\mathbb{Q}_a(N)$ contains an exception to (5.1). Then $N>2$, since if $\beta \in \mathbb{Q}_a(2)$, then $\beta$ is of the form $u/4+v\sqrt{a}/4$ with  integers $u,v$, and $M_{a,N}(\beta)= |u|+|v|$, so that
\begin{align*}
 \Delta_a(N)^2\mathcal{M}(\beta)\geq 4^2 \mathcal{M}(\beta) &\geq \frac{16}{16}(u^2+v^2)\geq c_9g[M_{a,N}(\beta)],
\end{align*} and if $N=1$, then $\beta$ is a rational integer, $\Delta_a(1)=a, M_{a,N}(\beta)=|\beta|$ and 
\begin{center}
 $\mathcal{M}(\beta)=M_{a,N}(\beta)^2 \geq c_9 g[M_{a,N}(\beta)]$.
\end{center} Let $p$ be the largest prime factor of $N$ and suppose that $p^L \mid N$ and let $N=pN_1$. Let $\zeta$ be a primitive $p^L$th root of unity. Now choose $\beta=\displaystyle\sum_{0\leq i,j \leq p-1}a_{ij}\zeta^i \sqrt[p^L]{a^j}$, $a_{ij}=\alpha_{ij}r_{ij}$ as in (2.9), to be an exception of (5.1), 
$\alpha_{ij} \in \mathbb{Q}_a(N_1)$ being algebraic integers. We use the abbreviations from the beginning of Section 4, and we choose $a_{ij}$ such that $\sum_{i,j}n_{ij}=n$. As a final piece of notation, we choose a positive number $c_{10}=c_{10}(k)$ such that
\begin{equation}
 \pi(t) < \dfrac{kt}{\log 2 \log t} \text{ whenever } t \geq c_{10},
\end{equation} $\pi(t)$ being the number of primes less than $t$. We now consider various cases.

\textit{First case.} $L \geq 2$.

By Lemma 2.2,
\begin{align*}
 \Delta_a(N)^2\mathcal{M}(\beta)&\geq \Delta_a(N)^2\sum_{i,j}\mathcal{M}(a_{ij})\\
 &\geq \Delta_a(N_1)^2\sum_{i,j}\mathcal{M}(\alpha_{ij})\\
 & \geq c_9\sum_{i,j}g(n_{ij}) \text{ since } N_1<N,\\
 & \geq c_9g(n), \text{ by Lemma 3.5},
\end{align*} and this contradicts the definition of $\beta$.

\textit{Second case.} $L=1$ and $p\geq \max \{c_5, 1 + (\log n)/p \}$.

By Lemma 2.1 and using the same argument as above,

\begin{align*}
 \Delta_a(N)^2\mathcal{M}(\beta)
 &\geq \Delta_a(N)^2\left(\displaystyle\sum_{j \neq 0, i \geq 0 }\mathcal{M}(a_{ij}) +\dfrac{1}{2(p-1)}\displaystyle\sum_{i,j}\mathcal{M}(a_{i0}- a_{j_0}) \right)\\
 &\geq \Delta_a(N_1)^2\left(\displaystyle\sum_{j \neq 0, i \geq 0 }\mathcal{M}(\alpha_{ij}) +\dfrac{1}{2(p-1)}\displaystyle\sum_{i,j}\mathcal{M}(\alpha_{i0}- \alpha_{j_0}) \right)\\ 
&\geq c_9\left(\displaystyle\sum_{j \neq 0, i \geq 0 }g(n_{ij}) +\dfrac{1}{2(p-1)}\displaystyle\sum_{i,j}g(n_{i0j0}) \right)\\ 
& \geq c_9 g(n), \end{align*} by Theorem 4.4 applied to $\sum_i a_{i0}\zeta^i$, by the identity \begin{center}${M_{a,N}}(.)_{|\mathbb{Q}_a(N_1)}=|\text{Nm}_{\mathbb{Q}_a(N)|\mathbb{Q}}(\text{disc}_{\mathbb{Q}_a(N)/ \mathbb{Q}_a(N_1)}\{ \zeta^i \sqrt[p]{a^j} \}_{i,j})|M_{a,N_1}(.)$,\end{center} 
and the fact that $g$ is increasing.

\textit{ Third case.} $L=1$ and $\max \{ \log c_1, c_5, c_{10} \} \leq p-1 \leq (\log n)/p$.

By Lemma 5.2,
\begin{align*}
\log (\Delta_a(N)^2\mathcal{M}(\beta)) &\geq \log n - \pi (p-1)\log 2\\
& \geq \log n - \pi (\log n) \log 2\\
& \geq \log n - \dfrac{k \log n}{\log \log n} \text{ by (5.3) }\\
& = \log f(n) \\
& > \log g(n) \text{ by Lemma 3.3.}
\end{align*}
\textit{Forth case.} $L=1$ and $p < \max \{ \log c_1, c_3, c_{10}  \}+1= c_{11}$.

Again, by Lemma 5.2,

\begin{align*}
\Delta_a(N)^2\mathcal{M}(\beta)&\geq 2^{-\pi (c_{11})} n\\
&\geq c_9 g(n) \text{ providing } c_9 \leq 2^{-\pi (c_{11})}.
\end{align*} So we have a contradiction in all cases if we choose $c_9 \leq 2^{-\pi (c_{11})}$.

\end{proof}

\begin{remark}
 As aways, let $\beta$ be an algebraic integer in  $\mathbb{Q}_a(N)$.
Letting $N=p_1^{e_1}...p_s^{e_s}$ be the decomposition in primes of $N$ with $p_1 <p_2< ...< p_s$, and $\zeta_m$ denotes a $m$-th primitive root of unity for each $ m  \mid N$, so that $\zeta_m=\zeta_N^{N/m}$, we choose $d_{i,t} \in \mathbb{Z}_{>0}, (1 \leq i \leq s, 1 \leq t \leq e_i)$ such that
\begin{center}
 $d_{i,t}\mathcal{O}_{\mathbb{Q}_a(p_1^{e_1}...p_i^t)} \subset \displaystyle\bigoplus_{l,k} \zeta_{p_i^t}^l \sqrt[p_i^t]{a^k}.\mathcal{O}_{\mathbb{Q}_a(p_1^{e_1}...p_i^{t-1})}$.
\end{center} (For example, $\text{Nm}_{\mathbb{Q}_a(N)|\mathbb{Q}}(\textit{disc}_{\mathbb{Q}_a(p_1^{e_1}...p_i^t)/\mathbb{Q}_a(p_1^{e_1}...p_i^{t-1})}(\{\zeta_{p_i^t}^l \sqrt[p_i^t]{a^k} \}_{l,k}))=d_{i,t}$ is a possible choice). 
Denoting $D_N:=\displaystyle\prod_{1 \leq i \leq s} \left(\displaystyle\prod_{1 \leq t \leq e_i} d_{i,t}\right)$, and $M_{a.N}(\beta)$ this time to be the smallest number of summands in a representation for $\beta$ of the form $\sum_{i,j}\xi_i\alpha_j/D_N$( allowing repetition), where $\xi_i$ is a root of unity and $\alpha_j$ is a positive real $n$-root of $a$,
we have that all the statements and proofs of this paper work ipsis literis with $D_N$ in place of $\Delta_a(N)$ for each $N$, and the new $M_{a,N}(\beta)$'s depending on the $D_N$'s instead of on the $\Delta_a(N)$'s.
In the case of $a=1$, since there are always integral basis formed entirely by roots of unity for the rings of integers of cyclotomic fields, one can take $d_{i,t},D_N$ all equal to $1$ always, and therefore our main Theorem and statements recover the classical result of Loxton in this case. 
\end{remark}
\begin{remark}
 Theorem 1.1 shows in particular that one can choose the roots of $a$ in the sum representing $\Delta_a(N)\beta$, so that the number of such roots is at most $L(\Delta_a(N)\house{\beta})$, where $L:\mathbb{R}_+ \rightarrow \mathbb{R}_+ $ is a suitable function, that one can take satisfying $L(x) \ll x^{2 + \varepsilon}$.  
\end{remark}

\end{document}